\documentclass[12pt,reqno]{amsart}
\usepackage{hyperref}
\usepackage{color, bm, amscd, tikz-cd}
\newcommand{\cB}{{\mathcal B}}

\newcommand{\cD}{{\mathcal D}}

\newcommand{\cF}{{\mathcal F}}

\newcommand{\cH}{{\mathcal H}}

\newcommand{\cK}{{\mathcal K}}

\newcommand{\bq}{{\mathbf q}}
\newcommand{\sbm}[1]{\left[\begin{smallmatrix} #1
		\end{smallmatrix}\right]}

\newcommand{\w}{{\omega}}
\newtheorem{thm}{Theorem}[section]

\newtheorem{corollary}[thm]{Corollary}
\newtheorem{lemma}[thm]{Lemma}

\theoremstyle{definition}

\newtheorem{definition}[thm]{Definition}
\newtheorem{remark}[thm]{Remark}

\newtheorem{example}[thm]{Example}

\numberwithin{equation}{section}

\def\textmatrix#1&#2\\#3&#4\\{\bigl({#1 \atop #3}\ {#2 \atop #4}\bigr)}
\def\dispmatrix#1&#2\\#3&#4\\{\left({#1 \atop #3}\ {#2 \atop #4}\right)}
\numberwithin{equation}{section}

\def\textmatrix#1&#2\\#3&#4\\{\bigl({#1 \atop #3}\ {#2 \atop #4}\bigr)}
\def\dispmatrix#1&#2\\#3&#4\\{\left({#1 \atop #3}\ {#2 \atop #4}\right)}

\begin{document}

\title[]{Models for $\bq$-commutative tuples of isometries}

\author[J. A. Ball]{Joseph A. Ball}
\address{Department of Mathematics, Virginia Tech, Blacksburg, VA 24061-0123, USA}
\email{joball@math.vt.edu}
\author[H. Sau]{Haripada Sau}
\address{Department of Mathematics, Indian Institute of Science Education and Research, Pashan, Pune, Maharashtra 411008, India} 
\email{hsau@iiserpune.ac.in, haripadasau215@gmail.com}
\thanks{The research of the second-named author is supported by DST-INSPIRE Faculty Fellowship DST/INSPIRE/04/2018/002458.}
\maketitle

\begin{abstract}
A pair of Hilbert space linear operators $(V_1,V_2)$ is said to be $q$-commutative, for a unimodular complex number $q$, if $V_1V_2=qV_2V_1$. A concrete functional model for $q$-commutative pairs of isometries is obtained. The functional model is parametrized by a collection of Hilbert spaces and operators acting on them. As a consequence, the collection serves as a complete unitary invariance for $q$-commutative pairs of isometries. A $q$-commutative operator pair $(V_1,V_2)$ is said to be doubly $q$-commutative, if in addition, it satisfies $V_2V_1^*=qV_1^*V_2$. Doubly $q$-commutative pairs of isometries are also characterized. Special attention is given to doubly $q$-commutative pairs of shift operators. The notion of $q$-commutativity is then naturally extended to the case of general tuples of operators to obtain a similar model for tuples of $q$-commutative isometries.
\end{abstract}
\section{Introduction}
A stepping stone to the development of model theory for contractive Hilbert space operators is what is known as the Wold decomposition: {\em every isometric operator $V$ acting on a Hilbert space $\cH$ is unitarily equivalent to the direct sum $S\oplus W$, where $W$ is a unitary operator and $S$ is a {\em shift} operator, i.e., $S$ is an isometry with $S^{*n}\to 0$, in the strong operator topology, as $n\to\infty$.} This is due to \cite{Halmos, Neumann} and \cite{Wold}. There has been numerous generalizations of this classical decomposition theorem. For example, see \cite{B-C-L, Slo1980} for development in the commutative setting and \cite{Sarkar2014, Slo1985} for doubly commutative setting; also see \cite{BBLPS, Burdak, GS, Pop1998, Pop2000, Pop2002, Pop2004, SZ} and references therein for results in this direction.

The objective of this paper is to further extend these decomposition results in the $q$-commutative and doubly $q$-commutative settings.
\begin{definition}
A pair $(V_1,V_2)$ of operators is said to be {\em$q$-commutative}, if
$$
V_1V_2=qV_2V_1.
$$
\end{definition}
Such pairs seem to be of significant importance in the area of quantum theory, see \cite{Connes, Majid, Prug}. Recently, $q$-commutative operators have been studied by some operator theorists. To mention some of these works, Bhat and Bhattacharyya \cite{Bhat-Bhattacharyya} studied $q$-commutative row contractions ($(T_1,T_2,\dots,T_d)$ (i.e., $T_iT_j=q(i,j)T_jT_i$ for each $i,j$ and $\sum_{i=1}^dT_iT_i^*\leq I$) in quest of its model. Later, Dey \cite{Dey} studied $q$-commutative row contractions for its dilation theory. In contrast to the consideration in this paper, $q(i,j)$ were allowed to be any non-zero complex numbers in both the papers \cite{Bhat-Bhattacharyya, Dey}. Recently, Keshari and Mallick \cite{KM2019} showed by a commutant lifting approach, that any $q$-commutative pair of contractive operators has a $q$-commutative unitary dilations, where $q$ is a unimodular complex number. Thus this is an extension of And\^o's dilation theorem \cite{ando} and that of Sebesty\'en \cite{Sebestyen}, where the result was proved for the case $q=-1$.

First, we note that unlike the commutative case, $q$-commutativity is `order-sensitive', i.e., if $(V_1,V_2)$ is $q$-commutative, then $(V_2,V_1)$ is $\overline q$-commutative. However, it follows from the definition that if $(V_1,V_2)$ is $q$-commutative, then so is $(V_1^*,V_2^*)$. For a concrete example of a $q$-commutative pair of isometries, let us choose a unimodular complex number $q$ and define the {\em rotation} operator $R_q$ on $H^2(\mathbb D^d)$, the Hardy space over the $d$-disk, as
\begin{align}\label{rotaion}
R_qf(\underline{z}):=f(q\underline{z})\text{ for all }f\in H^2(\mathbb D^d),
\end{align}where for $\underline{z}=(z_1,z_2,\dots,z_d)\in\mathbb D^d$, $q\underline{z}:=(qz_1,qz_2,\dots,qz_d)$. For each $j=1,2$, let $M_{z_j}$ denote the multiplication by `$z_j$' operator on $H^2(\mathbb D^2)$. Consider the pair on $H^2(\mathbb D^2)$
\begin{align}\label{Vj}
(V_1,V_2)=(R_{q}M_{z_1},M_{z_2}) \quad \text{or, } \quad (M_{z_1}R_{q},M_{z_2}).
\end{align}It is easy to verify that $(V_1,V_2)$ is a $q$-commutative pair of isometries. Let us note that if $R_q$ is the rotation operator on $H^2(\mathbb D)$ (simply denoted by $H^2$ in the sequel), then the rotation operator on $H^2(\mathbb D^d)$ is given by taking the $d$-fold tensor product of $R_q$. With a slight abuse of notation, {\em we use the same notation $R_q$ regardless of the dimension of the polydisk}. It follows easily that the rotation operator $R_q$ does not commute with $M_z$, the multiplication by `$z$' operator on $H^2$. Indeed, for every $f\in H^2$,
$$
R_qM_zf(z)=qzf(qz)=qM_zR_qf(z).
$$Thus $(R_q,M_z)$ is actually $q$-commutative.

For a Hilbert space $\cH$, the standard notation $\cB(\cH)$ is used to denote the algebra of bounded linear operators on $\cH$. Among several generalizations of the classical Wold decomposition, perhaps the most appealing is the one obtained by Berger, Coburn and Lebow \cite[Theorem 3.1]{B-C-L}. We extend the Berger-Coburn-Lebow program to the $q$-commutative setting: Our first main result shows that given Hilbert spaces $\cF$ and $\cK_u$, a projection $P$ and a unitary $U$ in $\cB(\cF)$, and a $q$-commutative pair of unitaries $(W_1,W_2)$ in $\cB(\cK_u)$, the pair
\begin{align}\label{Int:BCLmodel}
\left(\begin{bmatrix}R_q\otimes P^\perp U + M_zR_q\otimes P U &0\\0& W_1\end{bmatrix}, \begin{bmatrix}R_{\overline q}\otimes U^*P+R_{\overline q}M_z\otimes U^*P^\perp& 0\\0&W_2\end{bmatrix}\right)
\end{align}on $\sbm{H^2\otimes \cF\\\cK_u}$ is a $q$-commutative pair of isometries. And most importantly, for every $q$-commutative pair $(V_1,V_2)$ of isometries, there exists a collection $\{\cF,\cK_u;P,U,W_1,W_2\}$ of Hilbert spaces and operators as above such that $(V_1,V_2)$ is jointly unitarily equivalent to the model \eqref{Int:BCLmodel}. This is the content of Theorem \ref{T:qBCL}. Moreover, the correspondence between $q$-commutative pairs of isometries and the parameters $\{\cF,\cK_u;P,U,W_1,W_2\}$ is one-to-one in the sense explained in Theorem \ref{T:uniqueness}.

Recall that a commutative pair $(V_1,V_2)$ is said to be doubly commutative, if in addition, $V_2V_1^*=V_1^*V_2$. Let $(W_1,W_2)$ be a $q$-commutative pair of unitaries, i.e., $W_1W_2=qW_2W_1$. On multiplying $W_1^*$ from left and right successively, we see that $q$-commutativity of $(W_1,W_2)$ is equivalent to $W_2W_1^*=qW_1^*W_2$. In view of this, the following definition comes as a natural analogue of double commutativity.
\begin{definition}\label{D:Doublycomm}
A $q$-commutative pair of operators $(V_1,V_2)$ is said to be {\em doubly $q$-commutative}, if in addition, it satisfies
\begin{align}\label{double-q}
V_2V_1^*=qV_1^*V_2.
\end{align}
\end{definition}
We remark that if $V_1$ and $V_2$ are isometries satisfying just $V_2V_1^*=qV_1^*V_2$, then an easy computation shows that $(V_1V_2-qV_2V_1)^*(V_1V_2-qV_2V_1)=0$ and thus $V_1V_2=qV_2V_1$. Thus condition \ref{double-q} implies $q$-commutativity of $(V_1,V_2)$, if $V_1,V_2$ are isometries. The pair $(V_1,V_2)$ where each $V_j$ is as defined in \eqref{Vj} is an example of a doubly $q$-commutative pairs of isometries on $H^2(\mathbb D^2)$. However, it can be shown that the same pair when restricted to the space $H^2(\mathbb D^2)\ominus\{\text{constants}\}$, is not doubly $q$-commutative; this is explained in \S \ref{S:Examples}, where we discuss several other simple examples to illustrate the model theory. Theorem \ref{T:qdoublycomm} characterizes doubly $q$-commutative pairs of isometries.

As an application of the model theory, we exhibit a passage between commutative and $q$-commutative pairs of isometries. Similarly, we exhibit a way to go back and forth between the classes of doubly commutative and doubly $q$-commutative pairs of isometries. See Theorem \ref{T:comm&qcomm} and Theorem \ref{T:dcomm&dqcomm} for these connections. As a consequence of these correlations,  we show in Corollary \ref{C:DqCommShift} that given a doubly $q$-commutative pair of shift operators $(V_1,V_2)$, there is a unitary $\mathfrak s_q$ on $H^2(\mathbb D^2)$ such that $(V_1,V_2)$ is jointly unitarily equivalent to $(M_{z_1}\mathfrak s_q, \mathfrak s_{\overline q}M_{z_2})$ on $H^2(\mathbb D^2)$. This is an analogue of S\l oci\'nski \cite{Slo1980} who showed that every doubly commutative pair of shift operators is unitarily equivalent to $(M_{z_1},M_{z_2})$ on $H^2(\mathbb D^2)$.

The notion of $q$-commutativity is naturally extended to the case of general tuples of operators, see Definition \ref{D:tuplesq}. This model theory for the pair case is then applied to the case of a general $d$-tuple ($d>2$) of $q$-commutative isometries to obtain a similar model -- see Theorem \ref{T:bqBCL}. In \S \ref{S:Extension} we show that every $q$-commutative tuple of isometries $(X_1,X_2,\dots,X_d)$ can be extended to a $q$-commutative tuple of unitaries $(Y_1,Y_2,\dots, Y_d)$ (and hence doubly $q$-commutative) such that the unitary $Y_1Y_2\cdots Y_d$ is the minimal unitary extension of the isometry $X_1X_2\cdots X_d$. This both improves and gives a new proof of the `dilation' result of \cite{Jeu_Pinto} where it was shown that every doubly $q$-commutative tuple of isometries extends to a doubly $q$-commutative tuple of unitaries.

\section{Functional models for $q$-commutative pairs of isometries}
We begin with the following lemma which will be used in our quest for a functional model of $q$-commutative pairs of isometries. For a contraction $T$, we use the following standard notations for the {\em defect operator} and the {\em defect space} of $T$:
\begin{align*}
    D_{T}=(I-T^*T)^{1/2} \quad \mbox{and}\quad \cD_{T^*}=\overline{\operatorname{Ran}}\;D_T.
\end{align*}
\begin{lemma}\label{justlikethat} Let $(V_1,V_2)$ be a $q$-commutative pair of isometries on a Hilbert space $\mathcal H$ and $V=V_1V_2$. Then
\begin{itemize}
\item[(i)]\begin{align}\label{spaceequality}
\left\{\sbm{D_{V^*_1}V_2^*\\ D_{V^*_2}}h: h\in\mathcal H \right\}=\sbm{\mathcal D_{V_1^*}\\\mathcal D_{V_2^*}}=\left\{\sbm{D_{V^*_1}\\ D_{V^*_2}V^*_1}h: h \in \mathcal H\right\};
 \end{align}
 \item[(ii)] the defect space $\cD_{V^*}$ is unitarily equivalent to $\sbm{\cD_{V_1^*}\\ \cD_{V_2^*}}$ via the unitary
 \begin{align}
     D_{V^*}h\mapsto \begin{bmatrix}
     D_{V_1^*}\\ D_{V_2^*}V_1^*
     \end{bmatrix}h;\quad\mbox{and}
 \end{align}
 \item[(iii)] for every $j\geq1$,
\begin{align}\label{Intj}
V^{*j} V_1=q^{j-1}V_2^*V^{*j-1}\text{ and }V^{*j}V_2=\overline q^jV_1^*V^{*j-1}.
\end{align}
\end{itemize}
\end{lemma}
\begin{proof}
We only establish the first equality in \eqref{spaceequality}, the proof of the second equality is similar. We use the general fact that if $V$ is an isometry, then $D_{V^*}$ is the projection onto ${\operatorname{Ran} V}^\perp$. Let $f\oplus g \in \mathcal D_{V_1^*}\oplus\mathcal D_{V_2^*}$ be such that
$$\langle D_{V^*_1}V_2^*h\oplus D_{V^*_2}h, f\oplus g \rangle=0 \text{ for all $h\in\mathcal H$}.$$
This is equivalent to $\langle D_{V^*_1}V_2^*h, f\rangle + \langle D_{V^*_2}h, g\rangle =0 \text{ for all $h\in\mathcal H$}$, or, equivalently, $\langle h, V_2f\rangle + \langle h, g\rangle =0 \text{ for all $h\in\mathcal H$}$. Consequently, $g=-V_2f$, which implies that $g=D_{V_2^*}g=-(I-V_2V_2^*)V_2f=0$ and since $V_1$ is an isometry $f$ must also be $0$. This proves (i).

For (ii), we note that
\begin{align}\label{key1}
D_{V^{*}}^2=I-VV^*&=I-V_1V_1^*+V_1V_1^*-V_1V_2V_2^*V_1^*= D_{V_1^*}^2+V_1D_{V_2^*}^2V_1^*\\\label{key2}
&=I-V_2V_2^*+V_2V_2^*-V_2V_1V_1^*V_2^*=V_2D_{V_1^*}^2V_2^*+D_{V_2^*}^2.
\end{align}
This implies that for every vector $h$ in $\cH$,
\begin{align}\label{aux6}
\|D_{V^*}h\|^2=\|D_{V_1^*}V_2^*h\|^2+ \|D_{V_2^*}h\|^2=\|D_{V_1^*}h\|^2+\|D_{V_2^*}V_1^*h\|^2.
\end{align}Therefore to show that $\cD_{V^*}$ is isomorphic to $\sbm{\cD_{V_1^*}\\ \cD_{V_2^*}}$, we can consider the map
\begin{align}\label{Lambda}
D_{V^*}h\mapsto \begin{bmatrix}D_{V_1^*}\\D_{V_2^*}V_1^*\end{bmatrix}h \text{ for every }h\in\cH.
\end{align}Note that this is an isometry by \eqref{aux6} and surjective by part (ii) of the lemma.

For the intertwining relations \eqref{Intj}, we see that for every $j\geq1$,
 \begin{align}\label{aux2}
V^{*j} V_1=(V_2^*V_1^*)^jV_1=V_2^*(V_1^*V_2^*)^{j-1}=q^{j-1}V_2^*V^{*j-1}
\end{align} and
\begin{align}\label{aux3}
V^{*j} V_2=(V_2^*V_1^*)^jV_1=\overline q^j(V_1^*V_2^*)^jV_2=\overline q^jV_1^*(V_2^*V_1^*)^{*j-1}=\overline q^j V_1^*V^{* j-1}.
\end{align}
This completes the proof of the lemma.
\end{proof}
Now for the main theorem of this section, let us recall that the rotation operator $R_q$ is the unitary defined on $H^2$ as
$$
R_q:f(z)\mapsto f(qz).
$$
\begin{thm}\label{T:qBCL}
Let $V_1$ and $V_2$ be two operators acting on a Hilbert space $\cH$. Then the following are equivalent.
 \begin{itemize}
 \item[(i)] {\bf $q$-commutativity:} The pair $(V_1,V_2)$ is $q$-commutative;
 \item[(ii)] {\bf BCL-1 $q$-model:} There exist Hilbert spaces $\cF$ and $\cK_u$, a projection $P$ and a unitary $U$ in $\cB(\cF)$, and a pair $(W_1,W_2)$ of $q$-commuting unitaries in $\cB(\cK_u)$ such that $(V_1,V_2)$ is unitarily equivalent to
\begin{align}\label{iso-model1}
\left(\sbm{R_q\otimes P^\perp U + M_zR_q\otimes P U &0\\0& W_1}, \sbm{R_{\overline q}\otimes U^*P+R_{\overline q}M_z\otimes U^*P^\perp& 0\\0&W_2}\right) \text{ on }\sbm{H^2\otimes \cF\\\cK_u}.
\end{align}Moreover, the tuple $(\cF,\cK_u;P,U,W_1,W_2)$ can be chosen to be such that
\begin{align}\label{ExplctTuple1}
\begin{cases}
&\cF=\sbm{\cD_{V_1^*}\\\cD_{V_2^*}}, \quad \cK_u=\bigcap_{n\geq 0}(V_1V_2)^n\cH,\quad P:\sbm{f\\g}\mapsto\sbm{f\\0},\\
& U: \sbm{D_{V_1^*}\\D_{V_2^*}V_1^*}\mapsto \sbm{D_{V_1^*}V_2^*\\D_{V_2^*}}\text{ and } (W_1,W_2)=(V_1,V_2)|_{\cK_u},
\end{cases}
 \end{align}and the unitary operator $\tau_{\rm BCL}:\cH\to\sbm{H^2\otimes \cF\\\cK_u}$ can be chosen to be such that
 \begin{align}\label{tau'}
\tau_{\rm BCL} h= \begin{bmatrix}
D_{V_1^*}\\ D_{V_2^*}V_1^*
\end{bmatrix}(I-zV^*)^{-1}h\oplus\lim_{n\to\infty}(V_1V_2)^n(V_2^*V_1^*)^n h;
 \end{align}
 \item[(iii)] {\bf BCL-2 $q$-model:} There exist Hilbert spaces $\cF_\dag$ and $\cK_{u\dag}$, a projection $P_\dag$ and a unitary $U_\dag$ in $\cB(\cF_\dag)$, and a pair $(W_{1\dag},W_{2\dag})$ of $q$-commuting unitaries in $\cB(\cK_{u\dag})$ such that $(V_1,V_2)$ is unitarily equivalent to
\begin{align}\label{iso-model2}
\left(\sbm{R_q\otimes U_\dag^*P_\dag^\perp + M_zR_q\otimes U_\dag^*P_\dag&0\\0&W_{1\dag}}, \sbm{R_{\overline q}\otimes P_\dag U_\dag+R_{\overline q}M_z\otimes P_\dag^\perp U_\dag&0\\0&W_{2\dag}}\right) \text{ on }\sbm{H^2\otimes \cF_\dag\\\cK_{u\dag}}.
\end{align}Moreover, the tuple $(\cF_\dag,\cK_{u\dag};P_\dag,U_\dag,W_{1\dag},W_{2\dag})$ can be chosen to be such that
\begin{align}\label{ExplctTuple2}
(\cF_\dag,\cK_{u\dag};P_\dag,U_\dag,W_{1\dag},W_{2\dag})=(\cF,\cK_u;P,U^*,W_1,W_2),
\end{align}where $(\cF,\cK_u;P,U,W_1,W_2)$ is as in item (i) above, and the unitary operator $\tau_\dag:\cH\to \sbm{H^2\otimes \cF_\dag\\\cK_{u\dag}}$ can be chosen to be as in \eqref{tau'}.
\end{itemize}
\end{thm}
\begin{proof}[Proof of $(i)\Leftrightarrow (ii)$] We first show that the pair in \eqref{iso-model1} is a $q$-commuting pair of isometries. To that end, let $\xi$ be in $\cF$ and $n$ be a non-negative integer. We see that
\begin{align*}
(R_{\overline q}\otimes U^*P+R_{\overline q}M_z\otimes U^*P^\perp)(z^n\otimes \xi)=\overline q^nz^n\otimes U^*P \xi +\overline q^{n+1}z^{n+1}\otimes U^*P^\perp\xi
\end{align*}and therefore
\begin{align}
&(R_{ q}\otimes P^\perp U + M_zR_{ q}\otimes P U)( R_{\overline q}\otimes U^*P+R_{\overline q}M_z\otimes U^*P^\perp)(z^n\otimes \xi)\notag\\
=&(R_{ q}\otimes P^\perp U + M_zR_{ q}\otimes P U)(\overline q^nz^n\otimes U^*P \xi +\overline q^{n+1}z^{n+1}\otimes U^*P^\perp\xi)\notag\\
=&z^{n+1}\otimes P^\perp \xi + z^{n+1}\otimes P\xi =(M_z\otimes I_{\cF})(z^n\otimes \xi).\label{aux1}
\end{align}On other hand,
\begin{align*}
(R_{ q}\otimes P^\perp U + M_zR_{ q}\otimes P U)(z^n\otimes \xi)= q^nz^n\otimes P^\perp U \xi + q^nz^{n+1}\otimes PU\xi
\end{align*}and hence
\begin{align}
\notag &( R_{\overline q}\otimes U^*P+R_{\overline q}M_z\otimes U^*P^\perp)(R_{ q}\otimes P^\perp U + M_zR_{ q}\otimes P U)(z^n\otimes \xi)\\
\notag=&( R_{\overline q}\otimes U^*P+R_{\overline q}M_z\otimes U^*P^\perp)( q^nz^n\otimes P^\perp U \xi + q^nz^{n+1}\otimes PU\xi)\\\label{forgot}
=& \overline qz^{n+1}\otimes U^*PU\xi + \overline qz^{n+1}\otimes U^*P^\perp U\xi=\overline q(M_z\otimes I_{\cF})(z^n\otimes\xi).
\end{align}
From equations \eqref{aux1} and \eqref{forgot} therefore follows the $q$-commutativity of the BCL-1 $q$-model \eqref{iso-model1}. It remains to show that the entries of the BCL-1 $q$-model are isometries. But this is clear because the BCL-1 $q$-model is of the form
 $$
 \left( \begin{bmatrix}M_{(P^\perp + z P)U}&0 \\0&W_{1}\end{bmatrix}\begin{bmatrix}R_q&0\\0&I_{\cK_u}\end{bmatrix},
 \begin{bmatrix}R_{\overline q}&0\\0&I_{\cK_u}\end{bmatrix}\begin{bmatrix}M_{U^*(P + z P^\perp)} &0\\0& W_{2}\end{bmatrix}\right),
 $$and that the operators (neither $q$-commutative nor $q$-commutative)
 $$
\begin{bmatrix}M_{(P^\perp + z P)U}&0 \\0&W_{1}\end{bmatrix}\quad \mbox{and}\quad\begin{bmatrix}M_{U^*(P + z P^\perp)} &0\\0& W_{2}\end{bmatrix}
 $$
 are isometries. Now it follows from the fact that the product of an isometry and a unitary is always an isometry. Therefore $(ii)\Rightarrow(i)$.

We now establish the direction $(i)\Rightarrow(ii)$. Let us denote the isometry $V:=V_1V_2=qV_2V_1$. By Wold decomposition $V$ is unitarily equivalent to
 $$
 \begin{bmatrix}
 M_z&0\\0&W
 \end{bmatrix}:\begin{bmatrix}H^2(\cD_{V^*})\\ \cK_u\end{bmatrix}\to \begin{bmatrix}H^2(\cD_{V^*})\\ \cK_u\end{bmatrix}
 $$via the unitary
 \begin{align}\label{WoldUnitary}
     h\mapsto 
     \begin{bmatrix}
     D_{V^*}(I-zV^*)^{-1}h\\
     \lim_{n} V^nV^{*n}h
     \end{bmatrix}.
 \end{align}Here $W$ is a unitary operator on $\cK_u=\cap_{n\geq0}V^n\cH$. We first note that the subspace $\cK_u$ is invariant under both $V_1$ and $V_2$. We make use of the following {\em$q$-intertwining} relations, which are easy to establish:
 $$
V_1V^n=q^nV^nV_1 \text{ and }V_2V^n=\overline q^nV^n V_2 \text{ for every }n\geq 1.
 $$Let us suppose that for every $n\geq 0$, $g=V^nh_n$ for some $h_n\in\cH$. Then
 \begin{align*}
 V_1g=V_1 V^nh_n=V^n(q^nV_1h_n) \text{ and }
 V_2g=V_2V^nh_n=V^n(\overline q^nV_2h_n).
 \end{align*}Therefore $\cK_u$ is jointly $(V_1,V_2)$-invariant. So for each $j=1,2$, the avatar of $V_j$ on  $H^2(\cD_{V^*})\oplus \cK_u$ is of the form
 $$
 \tilde V_j=\begin{bmatrix}V^j_{11}&0\\ V^j_{21}&V^j_{22}\end{bmatrix}.
 $$ Note that $(V^1_{22},V^2_{22})$ is a pair of $q$-commuting isometries such that
 $$
W= V^1_{22}V^2_{22}=qV^2_{22}V^1_{22}.
 $$Since $W$ is a unitary, the pair $(V_{22}^1,V^2_{22})$ must be a $q$-commutative pair of unitaries -- a fact that trivially follows from \eqref{aux6} when applied to the pair $(V_{22}^1,V^2_{22})$. Therefore, each $\tilde V_j$ must be a block diagonal matrix. Consequently, it is enough to assume -- as we do for the rest of the proof -- that $V=V_1V_2$ is a shift. Therefore the operator $\tau_{\rm BCL}:\cH\to H^2\left(\sbm{\cD_{V_1^*}\\\cD_{V_2^*}}\right)$ defined as
\begin{align}\label{tau}
\tau_{\rm BCL} h= \sbm{D_{V_1^*}\\ D_{V_2^*}V_1^*}h +z\sbm{D_{V_1^*}\\ D_{V_2^*}V_1^*}V^*h+z^2\sbm{D_{V_1^*}\\ D_{V_2^*}V_1^*}V^{* 2}h+\cdots
\end{align}is a unitary and satisfies $\tau_{\rm BCL} V=M_z\tau_{\rm BCL}$.
To establish the unitary equivalence in part (ii) of the theorem, we use \eqref{Intj} to first note that for every $h\in\cH$
\begin{align*}
\tau_{\rm BCL} V_1h= \sbm{D_{V_1^*}\\ D_{V_2^*}V_1^*}V_1h +z\sbm{D_{V_1^*}\\ D_{V_2^*}V_1^*}V^*V_1h+z^2\sbm{D_{V_1^*}\\ D_{V_2^*}V_1^*}V^{* 2}V_1h+\cdots
\end{align*} is the same as
\begin{align*}
\sbm{0\\D_{V_2^*}}h+z\sbm{D_{V_1^*}V_2^*\\ D_{V_2^*}V_1^*V_2^*}h+qz^2\sbm{D_{V_1^*}V_2^*\\ D_{V_2^*}V_1^*V_2^*} V^*h+q^2z^3\sbm{D_{V_1^*}V_2^*\\ D_{V_2^*}V_1^*V_2^*} V^{*2}h+\cdots,
\end{align*}which we split in two parts as
\begin{align*}
&\left(\sbm{0\\D_{V_2^*}}h+qz\sbm{0\\D_{V_2^*}}V^*h +q^2z^2\sbm{0\\D_{V_2^*}}V^{*2}h+\cdots\right) \\
&+z\left( \sbm{D_{V_1^*}V_2^*\\0}h+qz \sbm{D_{V_1^*}V_2^*\\0}V^*h+ (qz)^2\sbm{D_{V_1^*}V_2^*\\0}V^{*2}h+\cdots\right),
\end{align*}which is equal to $(R_q\otimes P^\perp U+M_zR_q\otimes PU)\tau h$, where $P$ and $U$ are as describe in \eqref{ExplctTuple1} because for every $h\in\cH$
$$
P^\perp U\sbm{D_{V_1^*}\\ D_{V_2^*}V_1^*}h=P^\perp\sbm{D_{V_1^*}V_2^*\\ D_{V_2^*}}h=\sbm{0\\ D_{V_2^*}h} \text{ and }P U\sbm{D_{V_1^*}\\ D_{V_2^*}V_1^*}h=\sbm{D_{V_1^*}V_2^*h\\ 0}.
$$
It remains to show that
$$
\tau_{\rm BCL} V_2=(R_{\overline q}\otimes U^*P+R_{\overline q}M_z\otimes U^*P^\perp)\tau_{\rm BCL}.
$$For this we again use the relations \eqref{Intj} to note that
\begin{align*}
&\tau_{\rm BCL} V_2h=\sbm{D_{V_1^*}\\ D_{V_2^*}V_1^*}V_2h +z\sbm{D_{V_1^*}\\ D_{V_2^*}V_1^*}V^*V_2h+z^2\sbm{D_{V_1^*}\\ D_{V_2^*}V_1^*}V^{* 2}V_2h+\cdots\\
=&(I_{H^2}\otimes U^*)\left(\sbm{D_{V_1^*}V_2^*\\D_{V_2^*}}V_2h+z\sbm{D_{V_1^*}V_2^*\\D_{V_2^*}}V^*V_2h+z^2\sbm{D_{V_1^*}V_2^*\\D_{V_2^*}}V^{* 2}V_2h+\cdots\right)\\
=&(I_{H^2}\otimes U^*)\left(\sbm{D_{V_1^*}\\0}h+\overline qz\sbm{D_{V_1^*}V_2^*\\D_{V_2^*}}V_1^*h+(\overline qz)^2\sbm{D_{V_1^*}V_2^*\\D_{V_2^*}}V_1^*V^*h+\cdots\right)
\end{align*}As before, we split the last term in two parts as
\begin{align*}
&(I_{H^2}\otimes U^*)\left(\sbm{D_{V_1^*}\\0}h+\overline q z\sbm{D_{V_1^*}\\0}V^*h+(\overline qz)^2\sbm{D_{V_1^*}\\0}V^{*2}h+\cdots\right)\\
+&(I_{H^2}\otimes U^*)\overline qz\left(\sbm{0\\D_{V_2^*}V_1^*}h+\overline q z\sbm{0\\D_{V_2^*}V_1^*}V^*h+(\overline qz)^2\sbm{0\\D_{V_2^*}V_1^*}V^{*2}h+\cdots\right)
\end{align*}which is essentially equal to $(R_{\overline q}\otimes U^*P+R_{\overline q}M_z\otimes U^*P^\perp)\tau_{\rm BCL} h$, where $P$ and $U$ are as describe in \eqref{ExplctTuple1}. This establishes the equivalence of $(i)$ and $(ii)$. The equivalence of $(i)$ with $(iii)$ can be established in a similar way.
\end{proof}
\begin{definition}
For a $q$-commutative pair of isometries the tuples $(\cF,\cK_u;P,U,W_1,W_2)$ as in item (i) and $(\cF_\dag,\cK_{u\dag};P_\dag,U_\dag,W_{1\dag},W_{2\dag})$  as in item (ii) of Theorem \ref{T:qBCL}, will be referred to as the {\em BCL-1} and {\em BCL-2 $q$-tuples} of $(V_1,V_2)$, respectively, and (as is indicated in the statement) the models as in \eqref{iso-model1} and \eqref{iso-model2} will be called the {\em BCL-1} and {\em BCL-2 $q$-models} of $(V_1,V_2)$, respectively.
\end{definition}
\begin{remark}\label{R:BCLtuples}
Note that the BCL-2 $q$-model can be obtained from the BCL-1 $q$-model by the following transformation of the BCL $q$-tuples
$$
(\cF_\dag,\cK_{u\dag};P_\dag,U_\dag,W_{1\dag},W_{2\dag})\mapsto(\cF,\cK_u;U^*PU,U^*,W_1,W_2).
$$This indicates that it is enough to work with either of the model.
\end{remark}

It was observed in \cite{B-C-L} that a commutative pair of isometries is uniquely determined by the data set $(\cF,\cK_u;P,U,W_1,W_2)$. The same remains true in the case of $q$-commutativity also.
\begin{thm}\label{T:uniqueness}
Let $(V_1,V_2)$ and $(V_1',V_2')$ be two $q$-commutative pairs of isometries with $(\cF,\cK_u;P,U,W_1,W_2)$ and $(\cF',\cK_u';P',U',W_1',W_2')$ as their respective BCL-1 $q$-tuples. Then $(V_1,V_2)$ and $(V_1',V_2')$ are unitarily equivalent if and only if there exist unitary operators $\w:\cF\to\cF'$ and $\w_u:\cK_u\to\cK_u'$ such that
\begin{align}\label{coin}
\w(P,U)=(P',U')\w \text{ and }\w_u(W_1,W_2)=(W_1',W_2')\w_u.
\end{align}The statement remains true in case of BCL-2 $q$-tuples also.
\end{thm}
\begin{proof}
The easier direction is the `if' part. Note that if \eqref{coin} is true, then the unitary
$$
\sbm{I_{H^2}\otimes \w&0\\0&\w_u}:\sbm{H^2\otimes\cF\\\cK_u}\to\sbm{H^2\otimes\cF'\\\cK_u'}
$$intertwines the BCL-1 (and BCL-2) $q$-models of $(V_1,V_2)$ and $(V_1',V_2')$. For the converse part, suppose that the BCL-1 $q$-models
\begin{align*}
(V_1,V_2)=\left( \sbm{M_{(P^\perp + z P)U}&0 \\0&W_{1}}\sbm{R_q&0\\0&I_{\cK_u}},\sbm{R_{\overline q}&0\\0&I_{\cK_u}}\sbm{M_{U^*(P + z P^\perp)} &0\\0& W_{2}}\right) \text{ on }\sbm{H^2(\cF)\\\cK_u},
\end{align*}and
\begin{align*}
(V_1',V_2')=\left( \sbm{M_{(P'^\perp + z P')U'}&0 \\0&W_{1}'}\sbm{R_q&0\\0&I_{\cK_u}},\sbm{R_{\overline q}&0\\0&I_{\cK_u}}\sbm{M_{U'^*(P' + z P'^\perp)} &0\\0& W_{2}'}\right) \text{ on }\sbm{H^2(\cF')\\\cK_u'}
\end{align*}are unitarily equivalent via, say,
$$\tau=\sbm{\tau'&\tau_{12}\\\tau_{21}&\w_u}:\sbm{H^2(\cF)\\\cK_u}\to\sbm{H^2(\cF')\\\cK_u'}.$$ Adopting the notations $W:=W_1W_2$ and $W'=W_1'W_2'$, we see that $\tau$ must satisfy
\begin{align}\label{gen}
\sbm{\tau'&\tau_{12}\\\tau_{21}&\w_u}\sbm{M_z&0\\0&W}=\tau V_1V_2=V_1'V_2'\tau=\sbm{M_z&0\\0&W'}\sbm{\tau'&\tau_{12}\\\tau_{21}&\w_u},
\end{align}
equivalently, $\tau$ must satisfy
\begin{align}\label{gen1}
\tau'M_z=\tau'M_z,\quad \w_u W=W'\w_u \text{ and}\\\label{gen2}
\tau_{12}W=M_z\tau_{12},\quad \tau_{21}M_z=W'\tau_{21}.
\end{align} We now use the general functional analysis result that if $X$ is any operator that satisfies $XU=M_zX$ for some unitary $U$, then $X=0$. Therefore from \eqref{gen2}, we see that $\tau_{12}=0$. Since $\tau$ is a unitary that satisfies \eqref{gen}, it must also satisfy
$$
\sbm{\tau'&0\\\tau_{21}&\w_u}\sbm{M_z^*&0\\0&W^*}=\sbm{M_z^*&0\\0&W'^*}\sbm{\tau'&0\\\tau_{21}&\w_u},
$$comparing the $(12)$-entries of which we get $\tau_{21}M_z^*=W'^*\tau_{21}$. Since $W'$ is unitary, $\tau_{21}=0$. Therefore the unitary $\tau$ reduces to the block diagonal matrix $\text{diag}(\tau',\w_u)$. From the first equation in \eqref{gen1} we see that $\tau'=I_{H^2}\otimes \w$ for some unitary $\w:\cF\to\cF'$. Remembering that $\tau$ intertwines $(V_1,V_2)$ and $(V_1',V_2')$, we readily have the second equality in \eqref{coin} and for the first equality we note that $w$ must satisfy
\begin{align*}
wP^\perp U=P'^\perp U'\w \text{ and } \w PU=P'U'\w.
\end{align*}Adding these two equations we get $\w U=U'\w$, which then implies that $\w P=P'\w$. The proof for the case of BCL-2 $q$-tuples is along the same line as above. This completes the proof.
\end{proof}
The rest of this section is devoted to finding a connection between commutativity and $q$-commutativity. Let $(V_1,V_2)$ be a $q$-commutative pair of isometries on $\cH$ such that $V=V_1V_2$ is a shift. Note that in this case the space $\cK_u$ in BCL-1 $q$-tuple will be zero, and hence by Theorem \ref{T:qBCL}, $(V_1,V_2)$ is unitarily equivalent to
\begin{align*}
( M_{(P^\perp + z P)U}R_q,R_{\overline q}M_{U^*(P + z P^\perp)}),
\end{align*}via the unitary similarity
\begin{align}\label{puretau}
\tau_{\rm BCL}:h\mapsto \sbm{D_{V_1^*}\\ D_{V_2^*}V_1^*}h +z\sbm{D_{V_1^*}\\ D_{V_2^*}V_1^*}V^*h+z^2\sbm{D_{V_1^*}\\ D_{V_2^*}V_1^*}V^{* 2}h+\cdots.
\end{align}Let us denote the unitary
\begin{align}\label{cq}
\mathfrak r_q:=\tau_{\rm BCL}^*R_q\tau_{\rm BCL}:\cH\to\cH.
\end{align}
To compute the unitary $\mathfrak r_q$ explicitly, proceed as follows. For $h,k\in\cH$,
\begin{align*}
\langle\mathfrak r_qh,k\rangle&=\langle \tau_{\rm BCL}^*R_q\tau_{\rm BCL} h,k\rangle =\langle\sbm{D_{V_1^*}\\ D_{V_2^*}V_1^*}(I-qzV^*)^{-1}h, \tau_{\rm BCL}k\rangle\\
&=\sum_{n\geq 0}q^n\langle\sbm{D_{V_1^*}\\ D_{V_2^*}V_1^*}V^{*n}h, \sbm{
D_{V_1^*}\\ D_{V_2^*}V_1^*}V^{*n}k\rangle\\
&=\sum_{n\geq 0}q^n\langle D_{V^*}V^{*n}h,D_{V^*}V^{*n}k\rangle \quad \mbox{[using Lemma \ref{justlikethat}, part (ii)]}\\
&=\sum_{n\geq 0}q^n\langle V^nD_{V^*}V^{*n}h,k\rangle.
\end{align*}Thus 
\begin{align}\label{rq}
    \mathfrak r_qh= D_{V^*}h +q VD_{V^*}V^*h+\cdots q^n V^nD_{V^*}V^{*n}h+\cdots .
\end{align}

As a consequence of this observation and Theorem \ref{T:qBCL}, we get the following connection between commutativity and $q$-commutativity of a pair of isometries.
\begin{thm}\label{T:comm&qcomm}
Let $V_1$ and $V_2$ be isometric operators such that $V=V_1V_2$ is a shift operator. Then with the unitary $\mathfrak r_q$ as defined in \eqref{cq},
\begin{enumerate}
\item $(V_1,V_2)$ is commutative if and only if $(V_1\mathfrak r_q,\mathfrak r_{\overline q} V_2)$ is $q$-commutative;
\item $(V_1,V_2)$ is $q$-commutative if and only if $(V_1\mathfrak r_{\overline q},\mathfrak r_{ q} V_2)$ is commutative.
\end{enumerate}
\end{thm}
\begin{proof}
We prove only part (1) because it implies part (2). Suppose $(V_1,V_2)$ is a commutative pair of isometries and $(\cF;P,U)$ is a BCL-1 tuple of $(V_1,V_2)$. Then applying Theorem \ref{T:qBCL} for the $q=1$ case,
\begin{align}\label{Int1}
\tau_{\rm BCL}(V_1,V_2)=(M_{(P^\perp+zP)U},M_{U^*(P+zP^\perp)})\tau_{\rm BCL}  
\end{align}
via the unitary similarity $\tau_{\rm BCL}$ as in \eqref{puretau} above. In view of \eqref{cq} and \eqref{Int1},
\begin{align*}
(M_{(P^\perp+zP)U}R_q,R_{\overline q}M_{U^*(P+zP^\perp)})
&=(M_{(P^\perp+zP)U}\tau_{\rm BCL}\mathfrak r_q\tau_{\rm BCL}^*,\tau_{\rm BCL}\mathfrak r_{\overline q}\tau_{\rm BCL}^*M_{U^*(P+zP^\perp)})\\
&=\tau_{\rm BCL}(V_1\mathfrak r_q,\mathfrak r_{\overline q} V_2)\tau_{\rm BCL}^*.
\end{align*}By the equivalence of $(1)$ and $(2)$ of Theorem \ref{T:qBCL}, the pair $$(M_{(P^\perp+zP)U}R_q,R_{\overline q}M_{U^*(P+zP^\perp)})$$ is $q$-commutative, and thus so is the pair $(V_1\mathfrak r_q,\mathfrak r_{\overline q} V_2)$.
\end{proof}
 In view of the fact that $(R_q,M_z)$ is $q$-commutative, the following is an abstract version of Theorem \ref{T:comm&qcomm}.
\begin{thm}
Suppose that $V_1$, $V_2$ are some operators acting on a Hilbert space $\cH$,
${\mathfrak r}$ is a unitary operator on $\cH$ such that for a uni-modular $q$,
\begin{align}\label{abstractR}
  {\mathfrak r}   V_1V_2   = q \cdot V_1V_2 {\mathfrak r}.
\end{align}
Then $(V_1,V_2)$ is commutative if and only if $(V_1 {\mathfrak r}, {\mathfrak r}^* V_2 )$ is $q$-commutative.
\end{thm}
\begin{proof}
Let us denote $(W_1,W_2)=(V_1 {\mathfrak r}, {\mathfrak r}^* V_2 )$. Suppose that $(V_1, V_2)$ is commutative and compute
$$
W_1W_2 = V_1 {\mathfrak r}^* {\mathfrak r}V_2  = V_1 V_2 =:V
 $$
 while
$$
W_2W_1 = {\mathfrak r}^* V_2 V_1 {\mathfrak r}= {\mathfrak r}^* V {\mathfrak r} =  \overline q \cdot V =  \overline q \cdot W_1W_2.
$$
So $(W_1,W_2)$ is $q$-commutative. Conversely suppose $W_1W_2=q\cdot W_2W_1$, i.e., $V_1V_2=q\cdot \mathfrak r^*V_2V_1\mathfrak r$. By \eqref{abstractR}, this is same as $q\cdot \mathfrak r^* V_1V_2 \mathfrak r= q\cdot \mathfrak r^* V_2V_1 \mathfrak r$. This implies $V_1V_2=V_2V_1$.
\end{proof}


\section{Doubly $q$-commutative pairs of isometries}
Let us recall that a $q$-commutative pair of operators $(V_1,V_2)$ is said to be {\em doubly $q$-commutative}, if in addition, it satisfies $V_2V_1^*=qV_1^*V_2$. Note that if $(V_1,V_2)$ is doubly $q$-commutative, then so is $(V_1^*,V_2^*)$. Then next result is a characterization of doubly $q$-commutative pairs of isometries.
\begin{thm}\label{T:qdoublycomm}
Let $(V_1,V_2)$ be a pair of $q$-commutative isometries with BCL-1 and BCL-2 $q$-tuples as $(\cF,\cK_u;P,U,W_1,W_2)$ and $(\cF_\dag,\cK_{u\dag};P_\dag,U_\dag,W_{1\dag},W_{2\dag})$, respectively. Then the following are equivalent:
\begin{enumerate}
\item $(V_1,V_2)$ is doubly $q$-commutative;
\item $PUP^\perp=0$; \text{ and}
\item $P_\dag^\perp U_\dag P_\dag=0.$
\end{enumerate}
\end{thm}
\begin{proof}
By Theorem \ref{T:qBCL}, we can assume without loss of generality that $(V_1,V_2)$ is either the BCL-1 $q$-model \eqref{iso-model1} or the BCL-2 $q$-model \eqref{iso-model2}; to prove $(1)\Leftrightarrow(2)$, we work with the BCL-1 $q$-model. Since, $q$-commutativity of a pair of unitaries implies its doubly $q$-commutativity, we disregard the {\em unitary part} $(W_1,W_2)$ in the model \eqref{iso-model1} and suppose that
\begin{align*}
(V_1,V_2)=\left(R_q\otimes P^\perp U + M_zR_q\otimes P U ,R_{\overline q}\otimes U^*P+R_{\overline q}M_z\otimes U^*P^\perp\right) \text{ on }H^2\otimes \cF.
\end{align*}
We shall make use of the following identities concerning the two operators $R_q$ and $M_z$ on $H^2$. We do not prove these relations as the proofs are elementary. For every $n\geq1$,
\begin{align*}
&R_{\overline q}M_zR_{\overline q} (z^n)=\overline q^{2n+1}z^n,\quad R_{\overline q}M_zR_{\overline q}M_z^*(z^n)=\overline q^{2n-1}z^n\\
&R_{\overline q}M_z^*R_{\overline q}(z^n)=\overline q^{2n-1}z^{n-1},\quad R_{\overline q}M_z^*R_{\overline q}M_z(z^n)=\overline q^{2n+1}z^n.
\end{align*}
With the above relations in mind, we compute
\begin{align*}
V_2V_1^*&=(R_{\overline q}\otimes U^*P+R_{\overline q}M_z\otimes U^*P^\perp)(R_{\overline q}\otimes U^*P^\perp + R_{\overline q}M_z^*\otimes U^*P)\\
&=R_{\overline q}^2\otimes U^*PU^*P^\perp +R_{\overline q}^2M_z^*\otimes U^*PU^*P+R_{\overline q}M_zR_{\overline q}\otimes U^*P^\perp U^*P^\perp \\&\quad+R_{\overline q}M_zR_{\overline q}M_z^*\otimes U^*P^\perp U^*P
\end{align*}and
\begin{align*}
V_1^*V_2&=(R_{\overline q}\otimes U^*P^\perp + R_{\overline q}M_z^*\otimes U^*P)(R_{\overline q}\otimes U^*P+R_{\overline q}M_z\otimes U^*P^\perp)\\
&=R_{\overline q}^2\otimes U^*P^\perp U^*P +R_{\overline q}^2M_z\otimes U^*P^\perp U^*P^\perp +R_{\overline q}M_z^*R_{\overline q}\otimes U^*PU^*P\\
&\quad +R_{\overline q}M_z^*R_{\overline q}M_z\otimes U^*PU^*P^\perp.
\end{align*}
Suppose $n\geq1$ and $\xi \in\cF$. Then
\begin{align*}
V_2V_1^*(z^n\otimes \xi)&=\overline q^{2n}z^n\otimes U^*PU^*P^\perp\xi+\overline q^{2n-2}z^{n-1}\otimes U^*PU^*P\xi\\
&\quad+\overline q^{2n+1}z^{n+1}U^*P^\perp U^*P^\perp\xi+\overline q^{2n-1}z^n\otimes U^*P^\perp U^*P\xi
\end{align*}and
\begin{align*}
V_1^*V_2(z^n\otimes \xi)&=\overline q^{2n}z^n\otimes U^*P^\perp U^*P\xi+\overline q^{2n+2}z^{n+1}\otimes U^*P^\perp U^*P^\perp\xi\\
&\quad+\overline q^{2n-1}z^{n-1}\otimes U^*PU^*P\xi+\overline q^{2n+1}z^n\otimes U^*PU^*P^\perp \xi.
\end{align*}From the above expressions of $V_2V_1^*(z^n\otimes\xi)$ and $qV_1^*V_2(z^n\otimes \xi)$, one readily observes that
$$
V_2V_1^*(z^n\otimes\xi)=qV_1^*V_2(z^n\otimes \xi) \text{ whenever }n\geq 1 \text{ and }\xi\in\cF.
$$
We now compute
$$
V_2V_1^*(1\otimes \xi)=U^*PU^*P^\perp\xi+\overline q z\otimes U^*P^\perp U^*P^\perp\xi
$$and
$$
V_1^*V_2(1\otimes \xi)=U^*P^\perp U^*P\xi+\overline q^2z\otimes U^*P^\perp U^*P^\perp\xi +\overline qU^*PU^*P^\perp \xi.
$$Therefore $V_2^*V_1=qV_1^*V_2$ if and only if for every $\xi\in\cF$,
$$
V_2V_1^*(1\otimes \xi)=V_1^*V_2(1\otimes \xi),
$$which, in view of the above computation, is true if and only if
$$
PUP^\perp=0.
$$ This completes the proof of $(1)\Leftrightarrow(2)$. To complete the proof of the theorem, one can either work with the BCL-2 $q$-model in \eqref{iso-model2} and proceed as before to prove $(1)\Leftrightarrow(3)$, or, simply apply Remark \ref{R:BCLtuples} and establish the equivalence of $(2)$ and $(3)$.
\end{proof}
We wish to establish a connection between double commutativity and $q$-double commutativity in analogue of Theorem \ref{T:comm&qcomm}. We first observe the following.
\begin{lemma}\label{L:aux2}
The BCL-1 model
$$
(M_{(P^\perp + z P)U} \oplus W_{1}, M_{U^*(P + z P^\perp)} \oplus W_{2})
$$of a commutative pair of isometries is doubly commutative if and only if $PUP^\perp=0$.
\end{lemma}
\begin{proof}
Since a commuting pair of unitaries is automatically doubly commuting, we only investigate the doubly commutativity of the pair
$$
(V_1,V_2)=(I_{H^2}\otimes P^\perp U+ M_z\otimes PU , I_{H^2}\otimes U^*P + M_z\otimes U^*P^\perp).
$$
We note that
\begin{align*}
V_2^* V_1 & = ( I_{H^2} \otimes P  U  + M_z^*\otimes P^\perp U )
(I_{H^2}\otimes P^\perp U+ M_z\otimes PU) \\
& = I_{H^2}\otimes P  U {P^\perp} U   + M_z \otimes P  U  P  U + M_z^*\otimes{P^\perp}U{P^\perp}U  + I_{H^2}\otimes{P^\perp} U  P  U
\end{align*}
and
\begin{align*}
V_1 V_2^* & = (I_{H^2}\otimes P^\perp U+ M_z\otimes PU)(I_{H^2} \otimes P  U  + M_z^*\otimes P^\perp U)  \\
& = I_{H^2} \otimes {P^\perp}U P  U    +  M_z \otimes P  U  P  U
+ M_z^*\otimes {P^\perp}U{P^\perp}U   +  M_z M_z^* \otimes P  U {P^\perp} U  .
\end{align*}
From the above two expressions, we see after cancellation of common terms that
\begin{align*}
&V_2^* V_1 - V_1 V_2^*  =  (I - M_z M_z^*) \otimes  P  U {P^\perp} U .
\end{align*}
Since $ I_{H^2} - M_z M_z^*$ is the projection of $H^2$ on the constant functions in $H^2$, we see that $V_1$ double commutes with $V_2$ exactly when
$ P  U {P^\perp} U = 0$, or, equivalently, $ P  U {P^\perp} = 0$.
\end{proof}
\begin{thm}\label{T:dcomm&dqcomm}
Let $V_1$ and $V_2$ be isometries such that $V=V_1V_2$ is a shift, and $\mathfrak r_q$ be the unitary as in \eqref{cq}. Then
\begin{enumerate}
\item $(V_1,V_2)$ is doubly commutative if and only if $(V_1\mathfrak r_q,\mathfrak r_{\overline q} V_2)$ is doubly $q$-commutative;
\item $(V_1,V_2)$ is doubly $q$-commutative if and only if $(V_1\mathfrak r_{\overline q},\mathfrak r_{q} V_2)$ is doubly commutative.
\end{enumerate}
\end{thm}
\begin{proof}
The proof is similar to that of Theorem \ref{T:comm&qcomm}. For part (1), suppose $(V_1,V_2)$ is a commutative and $(\cF;P,U)$ is a BCL-1 tuple of $(V_1,V_2)$. By Theorem \ref{T:qBCL}, 
\begin{align}\label{Int1'}
\tau_{\rm BCL}(V_1,V_2)=(M_{(P^\perp+zP)U},M_{U^*(P+zP^\perp)})\tau_{\rm BCL}.
\end{align}where $\tau_{\rm BCL}:\cH\to H^2(\cF)$ is the unitary as in \eqref{puretau}.  Suppose that $(V_1,V_2)$ is doubly commutative. Hence by Lemma \ref{L:aux2}, we have $PUP^\perp=0$. By Theorem \ref{T:qdoublycomm}, this is equivalent to the BCL-1 $q$-model $(M_{(P^\perp+zP)U}R_q,R_{\overline q}M_{U^*(P+zP^\perp)})$ being doubly $q$-commutative. But as observed in the proof of Theorem \ref{T:comm&qcomm},
\begin{align*}
(M_{(P^\perp+zP)U}R_q,R_{\overline q}M_{U^*(P+zP^\perp)})=(V_1'\tau_{\rm BCL}\mathfrak r_q\tau_{\rm BCL}^*,\tau_{\rm BCL}\mathfrak r_{\overline q}\tau_{\rm BCL}^*V_2')=\tau_{\rm BCL}(V_1\mathfrak r_q,\mathfrak r_{\overline q} V_2)\tau_{\rm BCL}^*.
\end{align*}Therefore equivalently, the pair $(V_1\mathfrak r_q,\mathfrak r_{\overline q} V_2)$ must also be doubly $q$-commutative. Now part (1) implies part (2) and therefore the proof is complete.
\end{proof}
S\l oci\'nski \cite{Slo1980} proved that any pair of doubly commuting shift operators is unitarily equivalent to $(M_{z_1},M_{z_2})$ on $H^2(\mathbb D^2)$. As a corollary to Theorem \ref{T:dcomm&dqcomm}, we get the following analogue of S\l oci\'nski's result in the $q$-commutative setting.
\begin{corollary}\label{C:DqCommShift}
A pair of shift operators $(V_1,V_2)$ is doubly $q$-commutative if and only if it is unitarily equivalent to $(M_{z_1}\mathfrak s_q, \mathfrak s_{\overline q}M_{z_2})$ on $H^2(\mathbb D^2)$ for some unitary $\mathfrak s_q$ on $H^2(\mathbb D^2)$.
\end{corollary}
\begin{proof}
Suppose that $(V_1,V_2)$ is doubly $q$-commutative pair of shift operators. It is a general fact that if $(V_1,V_2)$ is $q$-commutative pair of isometries with one of the entries a shift, then the product $V=V_1V_2$ is also a shift. To see this we shall use the general fact that if {\em if $(T_1,T_2)$ is $q$-commutative, then with $T=T_1T_2$,
\begin{align}\label{aux7}
T^n=\overline q^{x_n}T_1^nT_2^n=q^{y_n}T_2^nT_1^n\text{ for every }n\geq1,
\end{align} where the sequences $\{x_n\}_{n\geq1}$ and $\{y_n\}_{n\geq1}$ are given by the iterative relations
 $$
 x_1=0,\;x_n=x_{n-1}+n-1, \text{ and }y_1=1,\; y_n=y_{n-1}+n.
 $$}
We omit the proof of \eqref{aux7} as it is routine. Applying this fact to the $q$-commutative pair $(V_1,V_2)$ of shift operators, we see that
$$
V^{*n}=q^{x_n}V_2^{*n}V_1^{*n}\to 0 \text{ as }n\to\infty.
$$Invoking part (2) of Theorem \ref{T:dcomm&dqcomm} we get $(V_1\mathfrak r_{\overline q},\mathfrak r_{q} V_2)$ is doubly commutative, where $\mathfrak r_q$ is the unitary as in \eqref{cq}. By S\l oci\'nski's characterization of doubly commutative pair of shifts, there exist a unitary $\tau_S:\cH\to H^2(\mathbb D^2)$ such that
 \begin{align}\label{DiskBidisk}
 \tau_S(V_1\mathfrak r_{\overline q},\mathfrak r_{q} V_2)=(M_{z_1},M_{z_2})\tau_S.
 \end{align}The first component of \eqref{DiskBidisk} gives $\tau_S V_1=M_{z_1}\tau_S \mathfrak r_q=M_{z_1}\tau_S \mathfrak r_q\tau_S^*\cdot \tau_S$. This and a similar treatment for the second component give
$$
\tau_S(V_1,V_2)=(M_{z_1}\tau_S \mathfrak r_{ q}\tau_S^*,\tau_S\mathfrak r_{\overline q}\tau_S^*M_{z_2})\tau_S,
$$which readily implies that with
$$
\mathfrak s_q:=\tau_S \mathfrak r_{ q}\tau_S^*=\tau_S\tau_{\rm BCL}^*R_q\tau_{\rm BCL}\tau_S^*
$$the doubly $q$-commutative pair $(V_1,V_2)$ is unitarily equivalent to $(M_{z_1}\mathfrak s_q,\mathfrak s_{\overline q}M_{z_2})$.

\end{proof}

\section{Examples}\label{S:Examples}
It is interesting to work with some concrete examples to illustrate the model theory. First we exhibit a simple example of a pair of isometric operators that is doubly $q$-commutative.
\begin{example}
Consider the pair $(V_1,V_2)=(R_q,M_z)$ on the Hardy space $H^2$. We have seen in the introduction that this pair is $q$-commutative. To see that this is doubly $q$-commutative, we prove the general fact that if a pair $(T_1,T_2)$ is $q$-commutative and $T_1$ is unitary, then it is doubly $q$-commutative. For this we simply multiply $T_1T_2=qT_2T_1$ by $T_1^*$ from right and left successively, to get $T_2T_1^*=qT_1^*T_2$. It is interesting to note that if, instead of $T_1$, $T_2$ is unitary, then $(T_1,T_2)$ would be doubly $\overline q$-commutative.  Below we illustrate the equivalence of $(1)$ and $(2)$ of Theorem \ref{T:qBCL} for this particular example. First we compute explicitly the BCL-1 $q$-tuple for this pair.

Let us first note that if $(V_1,V_2)=(R_q,M_z)$, then
$$
\sbm{D_{V_1^*}\\D_{V_2^*}}=\sbm{0\\P_{\mathbb C}}:\sbm{H^2\\H^2}\to\sbm{H^2\\H^2}, \text{ and therefore }\cF=\sbm{\cD_{V_1^*}\\\cD_{V_2^*}}=\sbm{0\\\mathbb C},
$$where $P_{\mathbb C}$ is the orthogonal projection of $H^2$ onto the constant functions. Let $f(z)=a_0+za_1+\cdots+z^na_n+\cdots$ be in $H^2$. We note that
$$
D_{V_1^*}V_2^*f=0 \text{ and }D_{V_2^*}V_1^*f=D_{M_z^*}R_{\overline q}=a_0.
$$Therefore
$$
U:\sbm{D_{V_1^*}\\D_{V_2^*}V_1^*}=\sbm{0\\ a_1}\mapsto\sbm{0\\a_0}=\sbm{D_{V_1^*}V_2^*\\D_{V_2^*}}
$$is essentially $I_{\mathbb C^2}$. It is interesting to note that if $P$ is the projection of $\cD_{V_1^*}\oplus\cD_{V_2^*}$ onto $\cD_{V_1^*}$ (which is zero), then $P$ is essentially $\sbm{0&0\\0&0}$, while $P^\perp=\sbm{0&0\\0&1}$.  Since $(R_q,M_z)$ is $q$-commutative, applying \eqref{aux7} to the pair $(R_q,M_z)$, we get with $V=R_qM_z$
$$
V^{*n}f=\overline q^{y_n}R_{\overline q^n}{M_z^*}^{n}f=\overline q^{y_n}(a_{n},\overline q^na_{n+1},\overline q^{2n}a_{n+2},\dots).
$$Since $V=V_1V_2$ on $H^2$ is a shift operator, the general unitary identification $\tau_{\rm BCL}$ from $H^2$ onto $H^2\otimes \sbm{\cD_{V_1^*}\\\cD_{V_2^*}}$, which is of the form (as shown in \eqref{tau})
$$
\tau_{\rm BCL} h= \sbm{D_{V_1^*}\\ D_{V_2^*}V_1^*}h +z\sbm{D_{V_1^*}\\ D_{V_2^*}V_1^*}V^*h+z^2\sbm{D_{V_1^*}\\ D_{V_2^*}V_1^*}V^{* 2}h+\cdots
$$is given in this case as
\begin{align}\label{tauex1}
f\mapsto \sbm{0\\a_0}+ z\overline q\sbm{0\\a_1}+\cdots+z^n\overline q^{y_n}\sbm{0\\a_n}+\cdots.
\end{align}Therefore
\begin{align*}
(P^\perp U+M_zPU)R_q\tau_{\rm BCL} f(z)&=\sbm{0&0\\0&1}\left(\sbm{0\\a_0}+ z\overline q^2\sbm{0\\a_1}+\cdots+z^n\overline q^{y_n+n}\sbm{0\\a_n}+\cdots\right)\\
 &=\left(\sbm{0\\a_0}+ z\overline q^2\sbm{0\\a_1}+\cdots+z^n\overline q^{y_n+n}\sbm{0\\a_n}+\cdots\right)=R_q \tau_{\rm BCL} f(z).
\end{align*} Similar computation for the intertwining relation $R_{\overline q}(U^*P+M_zU^*P^\perp)\tau_{\rm BCL}=R_{\overline q}M_z\tau_{\rm BCL}$.

In view of Theorem \ref{T:qdoublycomm}, that the pair $(R_qM_z,M_z)$ is doubly $q$-commutative is reflected in the fact that
$$
PUP^\perp=\sbm{0&0\\0&0}\sbm{1&0\\0&1}\sbm{0&0\\0&1}=\sbm{0&0\\0&0}.
$$
\end{example}

Next we find an example of a pair of shift operators that is $q$-commutative but not doubly $q$-commutative.
\begin{example}
Consider the pair $(V_1,V_2)=(R_qM_z,M_z)$ on $H^2$. Then for every $f\in H^2$,
\begin{align*}
&V_1V_2f(z)=R_qM_z^2f(z)=q^2z^2f(qz) \text{ while},\\
&V_2V_1f(z)=M_zR_qM_zf(z)=M_zqzf(qz)=qz^2f(qz),
\end{align*}showing that $(V_1,V_2)$ is a $q$-commutative pair. However, it should be noted that the pair is not doubly $q$-commutative. One way to see this is that
\begin{align*}
V_2V_1^*(1)=M_zM_z^*R_{\overline q}(1)=0 \text{ but }V_1^*V_2(1)=M_z^*R_{\overline q}M_z(1)=\overline q.
\end{align*}Therefore $V_2V_1^*\neq q V_1^*V_2$. To see that $V_1=R_qM_z$ is actually a shift operator, we apply \eqref{aux7} to the $q$-commutative pair $(R_q,M_z)$ to note
$$
V_1^{*n}=(M_z^*R_{\overline q})^n=\overline q^{y_n}R_{\overline q^n}M_z^{*n}\to 0 \text{ in the strong operator topology as }n\to\infty.
$$Below we compute the BCL-1 $q$-tuple corresponding to the pair $(V_1,V_2)=(R_qM_z,M_z)$. Let us first note that $D_{V_1^*}=I-V_1V_1^*=I-R_qM_zM_z^*R_{\overline q}=R_qD_{M_z^*}R_{\overline q}$, which is essentially the same as $D_{M_{z^*}}=P_{\mathbb C}$. Therefore
$$
\sbm{D_{V_1^*}\\ D_{V_2^*}}=\sbm{P_{\mathbb C}\\P_{\mathbb C}}:\sbm{H^2\\H^2}\to\sbm{H^2\\H^2}, \text{ and therefore }\sbm{\cD_{V_1^*}\\\cD_{V_2^*}}=\sbm{\mathbb C\\\mathbb C}.
$$Let $f(z)=a_0+za_1+\cdots+z^na_n+\cdots$ be in $H^2$. We note that
$$
D_{V_1^*}V_2^*f=D_{M_z^*}M_z^*f=a_1 \text{ and }D_{V_2^*}V_1^*f=D_{M_z^*}M_z^*R_{\overline q}=\overline q a_1.
$$Therefore
$$
U:\sbm{D_{V_1^*}\\D_{V_2^*}V_1^*}=\sbm{a_0\\\overline q a_1}\mapsto\sbm{a_1\\a_0}=\sbm{D_{V_1^*}V_2^*\\D_{V_2^*}}
$$is given by $U=\sbm{0&q\\0&1}$. Next we note that since $(R_q,M_z)$ is $q$-commutative, $(R_q,M_z^2)$ is $q^2$-commutative, and therefore applying \eqref{aux7} we get with $V=V_1V_2$
$$
V^{*n}f=\overline q^{2y_n}R_{\overline q^n}{M_z^*}^{2n}f=\overline q^{2y_n}(a_{2n},\overline q^na_{2n+1},\overline q^{2n}a_{2n+2},\dots).
$$Since $V$ is a shift operator, the unitary identification $\tau_{\rm BCL}$ from $H^2$ onto $H^2\otimes \sbm{\cD_{V_1^*}\\\cD_{V_2^*}}$ in this case is given by
\begin{align}\label{tauex}
f\mapsto \sbm{a_0\\\overline qa_1}+ z\overline q^2\sbm{a_2\\\overline q^2a_3}+\cdots+z^n\overline q^{2y_n}\sbm{a_{2n}\\\overline q^{n+1}a_{2n+1}}+\cdots.
\end{align}To demonstrate that this $\tau_{\rm BCL}$ intertwines $(V_1,V_2)$ and $(M_{(P^\perp+zP)U}R_{ q},R_{\overline q}M_{U^*(P+zP^\perp)})$, we compute
\begin{align*}
P^\perp UR_q\tau_{\rm BCL} f&=P^\perp U\left(\sbm{a_0\\\overline qa_1}+ z\overline q\sbm{a_2\\\overline q^2a_3}+\cdots+z^n\overline q^{2y_n-n}\sbm{a_{2n}\\\overline q^{n+1}a_{2n+1}}+\cdots\right)\\
&=P^\perp\left(\sbm{a_1\\a_0}+ z\overline q\sbm{\overline qa_3\\a_2}+\cdots+z^n\overline q^{2y_n-n}\sbm{\overline q^{n}a_{2n+1}\\a_{2n}}+\cdots\right)\\
&=\sbm{0\\a_0}+ z\overline q\sbm{0\\a_2}+\cdots+z^n\overline q^{2y_n-n}\sbm{0\\a_{2n}}+\cdots
\end{align*}and (using the action of $UR_q\tau_{\rm BCL} f$ from the above computation)
\begin{align*}
M_zPUR_q\tau_{\rm BCL} f&=M_zP\left(\sbm{a_1\\a_0}+ z\overline q\sbm{\overline qa_3\\a_2}+\cdots+z^n\overline q^{2y_n-n}\sbm{\overline q^{n}a_{2n+1}\\a_{2n}}+\cdots\right)\\
&=z\sbm{a_1\\0}+ z^2\overline q\sbm{\overline qa_3\\0}+\cdots+z^{n+1}\overline q^{2y_n-n}\sbm{\overline q^{n}a_{2n+1}\\0}+\cdots.
\end{align*}Therefore
\begin{align}\label{finex2}
(P^\perp U +M_zPU)R_q\tau_{\rm BCL} f&=\sbm{0\\a_0}+z\sbm{a_1\\\overline q a_2}+\cdots+ z^n\overline q^{2y_{n-1}}\sbm{a_{2n-1}\\\overline q^{n}a_{2n}}+\cdots.
\end{align}We note that
$$V_1f(z)=R_qM_zf=za_0q+z^2a_1q^2+\cdots+z^na_{n-1}q^n+\cdots=:\sum_{n\geq0}z^nb_n.$$Therefore replacing $f$ by $V_1f$ in the expression \eqref{tauex} of $\tau_{\rm BCL}$, we get $\tau_{\rm BCL} V_1 f$ the same as $(P^\perp U +M_zPU)R_q\tau_{\rm BCL} f$. Similar computation for the other intertwining relation.

In view of Theorem \ref{T:qdoublycomm}, that the pair $(R_qM_z,M_z)$ is not doubly $q$-commutative is reflected in the fact that
$$
PUP^\perp=\sbm{1&0\\0&0}\sbm{0&q\\1&0}\sbm{0&0\\0&1}=\sbm{0&q\\0&0}\neq \sbm{0&0\\0&0}.
$$
\end{example}
The following couple of examples are interesting to note.
\begin{example}\label{E:2nd}
Consider the pair $(V_1,V_2)=(R_qM_{z_1},M_{z_2})$ on $H^2(\mathbb D^2)$. We have noticed in the Introduction that this is indeed $q$-commuting. The computation below shows that it is actually doubly $q$-commutative.
\begin{align*}
&V_2V_1^*f(z_1,z_2)=M_{z_2}M_{z_1}^*R_{\overline q}f(z_1,z_2)=M_{z_2}M_{z_1}^*f(\overline qz_1,\overline qz_2)=M_{z_1}^*z_2f(\overline qz_1,\overline qz_2) \text{ and}\\
&V_1^*V_2f(z_1,z_2)=M_{z_1}^*R_{\overline q}M_{z_2}f(z_1,z_2)=\overline qM_{z_1}^*z_2f(\overline qz_1,\overline qz_2).
\end{align*}Consider the subspace $\cH_\diamond:=H^2(\mathbb D^2)\ominus \{\text{constants}\}$. Just like commutativity, it is trivial that $q$-commutativity property is hereditary, i.e., the restriction of a $q$-commutative pair is $q$-commutative. However, the restriction $(V_1',V_2')=(R_qM_{z_1},M_{z_2})|_{\cH_\diamond}$ is not doubly $q$-commutative as the following computation reveals:
$$
V_2'V_1'^*(z_1)=M_{z_2}M_{z_1}^*R_{\overline q}(z_1)=0\neq\overline qz_2=M_{z_1}^*\overline qz_1z_2=qM_{z_1}^*R_{\overline q}M_{z_2}(z_1) = qV_1'^*V_2'(z_1).
$$
\end{example}
It is interesting to have an example of a pair of isometries which is not $q$-commutative for any complex number $q$. Let $\alpha,\beta$ are two distinct numbers in $\mathbb T$. Consider
\begin{align*}
V_1=\sbm{\alpha&0\\0&\beta} \text{ and }V_2=\sbm{\frac{1}{\sqrt{2}}&\frac{1}{\sqrt{2}}\\\frac{1}{\sqrt{2}}&-\frac{1}{\sqrt{2}}}.
\end{align*}Then clearly
$$
V_1V_2=\sbm{\frac{1}{\sqrt{2}}\alpha&\frac{1}{\sqrt{2}}\alpha\\\frac{1}{\sqrt{2}}\beta&-\frac{1}{\sqrt{2}}\beta}\neq q\sbm{\frac{1}{\sqrt{2}}\alpha&\frac{1}{\sqrt{2}}\beta\\\frac{1}{\sqrt{2}}\alpha&-\frac{1}{\sqrt{2}}\beta}=qV_2V_1
$$for any number $q$, because $\alpha$ and $\beta$ are distinct.

\section{The tuple case}
In this section, we use the model for the pair case to exhibit a parallel model for tuples $(V_1,V_2,\dots,V_d)$ of $q$-commutative isometries. We first define $q$-commutativity for tuples of operators.
\begin{definition}\label{D:tuplesq}
Let $q:\{1,2,\dots,d\}\times\{1,2,\dots,d\}\to\mathbb T$ be a function such that $q(i,i)=1$ and $q(i,j)=\overline{q(j,i)}$ for each $i,j=1,2,\dots,d$. A $d$-tuple $(V_1,V_2,\dots,V_d)$ of operators is said to be {\em$q$-commutative}, if
$$
V_iV_j=q(i,j)V_jV_i \quad\mbox{for each }i,j=1,2,\dots,d.
$$
\end{definition}
As an example of a $q$-commutative tuple of isometries, let us define $V_j$ on $H^2(\mathbb D^d)$, the Hardy space of the $d$-disk, as
\begin{align}\label{TupleVj}
V_j=R_{q^{d-j}}M_{z_j} \text{ or } M_{z_j}R_{q^{d-j}} \text{ for each $j=1,2,\dots,d$,}
\end{align}and $q:\{1,2,\dots,d\}\times\{1,2,\dots,d\}\to\mathbb T$ as $q(i,j)=q^{j-i}$. To see that $(V_1,V_2,\dots,V_d)$ is $q$-commutative, we compute
\begin{align*}
V_iV_jf(\underline{z})=R_{q^{d-i}}M_{z_i}R_{q^{d-j}}M_{z_j}f(\underline{z})=q^{d-j}R_{q^{d-i}}z_iz_jf(q^{d-j}\underline{z})=q^{3d-2i-j}z_iz_jf(q^{2d-i-j}\underline{z})
\end{align*}while $V_jV_if(\underline{z})=q^{3d-i-2j}z_iz_jf(q^{2d-i-j}\underline{z})$ (obtained by just switching $(i,j)$ to $(j,i)$ in the above expression).

Let us denote
$$V_{(i)}:=V_1\cdots V_{i-1}V_{i+1}\cdots V_d.$$ A key observation that makes it possible to apply the results for the pair case to the general case, is that if $(V_1,V_2,\cdots,V_d)$ is $q$-commutative, then for each $i=1,2,\dots,d$, the pair $(V_i,V_{(i)})$ is $q_i$-commutative, where
\begin{align}\label{qi}
q_i:=\prod_{j=1}^d q(i,j).
\end{align}This is because for each $i$,
$$
V_iV_{(i)}=V_iV_1V_2\cdots V_{i-1}V_{i+1}\cdots V_d=\prod_{i\neq j=1}^dq(i,j)V_{(i)}V_i=q_iV_{(i)}V_i,
$$where we used the fact that $q(i,i)=1$.
This observation makes it easy to obtain a Berger--Coburn--Lebow-type model for any $q$-commutative tuples of isometries $(V_1,V_2,\dots,V_d)$. Indeed, the idea is to just apply Theorem \ref{T:qBCL} to each of the $q_i$-commutative pairs $(V_i,V_{(i)})$. However, unlike the pair case, a BCL-1 and BCL-2 $q$-models need not in general be $q$-commutative. This will happen when the BCL-1 and BCL-2 $q$-tuples satisfy some compatibility conditions.
\begin{thm}\label{T:bqBCL}
Let $(V_1,V_2,\dots, V_d)$ be a $d$-tuple of $q$-commutative isometries. Then
 \begin{enumerate}
 \item {\bf BCL-1 $q$-model:} there exist Hilbert spaces $\cF$ and $\cK_u$, projections $P_1, P_2,\dots,P_d$ and unitaries $U_1,U_2,\dots,U_d$ in $\cB(\cF)$, and a $q$-commutative tuple $(W_1,W_2,\dots,W_d)$ of unitaries in $\cB(\cK_u)$ such that for each $i=1,2,\dots,d$, $V_i$ is unitarily equivalent to
\begin{align}\label{iso-model1'}
\sbm{R_{q_i}\otimes P_i^\perp U_i + M_zR_{q_i}\otimes P_i U_i &0\\0& W_i} \text{ on }\sbm{H^2\otimes \cF\\\cK_u}
\end{align}and $V_{(i)}$ is unitarily equivalent to
\begin{align}
\sbm{R_{\overline q_i}\otimes U_i^*P_i+R_{\overline q_i}M_z\otimes U_i^*P_i^\perp& 0\\0&W_{(i)}}  \text{ on }\sbm{H^2\otimes \cF\\\cK_u}.
\end{align}Moreover, the tuple $(\cF,\cK_u;P_i,U_i,W_i)_{i=1}^d$ can be chosen to be such that
\begin{align}\label{ExplctTuple1'}
\begin{cases}
&\cF=\cD_{V_1^*}\oplus\cD_{V_2^*}\oplus\cdots\oplus\cD_{V_d^*}, \quad \cK_u=\bigcap_{n\geq 0}(V_1V_2\cdots V_d)^n\cH,\\
& (W_1,W_2,\dots,W_d)=(V_1,V_2,\dots,V_d)|_{\cK_u},\quad P_i= \text{projection onto }\cD_{V_i^*},\text{ and}\\
& U_i: D_{V_i^*}\oplus\Delta_iD_{V_{(i)}^*}V_i^*\mapsto D_{V_i^*}V_{(i)}^*\oplus\Delta_iD_{V_{(i)}^*} \text{ for some unitary}\\ &\Delta_i:\cD_{V_{(i)}^*}\to\oplus_{i\neq j=1}^d\cD_{V_j^*} \text{ given explicitly in }\eqref{Delta} \text{ below};
\end{cases}
 \end{align}and
 \item {\bf BCL-2 $q$-model:} there exist Hilbert spaces $\cF_\dag$ and $\cK_{u\dag}$, projections $P_{i\dag}$ and a unitary $U_{i\dag}$ in $\cB(\cF_\dag)$, and a tuple $(W_{1\dag},W_{2\dag},\dots,W_{d\dag})$ of $q$-commutative unitaries in $\cB(\cK_{u\dag})$ such that for each $i=1,2,\dots,d$, $V_i$ is unitarily equivalent to
\begin{align}\label{iso-model2'}
\sbm{R_{q_i}\otimes U_{i\dag}^*P_{i\dag}^\perp + M_zR_{q_i}\otimes U_{i\dag}^*P_{i\dag}&0\\0&W_{i\dag}} \text{ on }\sbm{H^2\otimes \cF_\dag\\\cK_{u\dag}}
\end{align}and $V_{(i)}$ is unitarily equivalent to
\begin{align}
\sbm{R_{\overline q_i}\otimes P_{i\dag} U_{i\dag}+R_{\overline q_i}M_z\otimes P_{i\dag}^\perp U_{i\dag}&0\\0&W_{(i)\dag}}  \text{ on }\sbm{H^2\otimes \cF_\dag\\\cK_{u\dag}}.
\end{align}Moreover, the tuple $(\cF_\dag,\cK_{u\dag};P_{i\dag},U_{i\dag},W_{i\dag})_{i=1}^d$ can be chosen to be such that
\begin{align}\label{ExplctTuple2'}
(\cF_\dag,\cK_{u\dag};P_{i\dag},U_{i\dag},W_{i\dag})=(\cF,\cK_u;P_i,U_i^*,W_i)\text{ for each }i,
 \end{align}where $(\cF,\cK_u;P_i,U_i,W_i)_{i=1}^d$ is as in part (1) above.
\end{enumerate}
\end{thm}
\begin{proof}
As in the pair case, we only do the analysis for part $(1)$, as a similar analysis works for part (2). The first step is to fix $i=1,2,\dots, d$ and apply the implication $(1)\Rightarrow(2)$ of Theorem \ref{T:qBCL} to the $q_i$-commutative pair $(V_i,V_{(i)})$. This will give us Hilbert spaces $\cF_i$, $\cK_{iu}$, a projection $P_i$, a unitary $U_i$ in $\cB(\cF_i)$, and a pair $(W_{i},W_{i}')$ of $q_i$-commuting unitaries in $\cB(\cK_{iu})$ such that $(V_i,V_{(i)})$ is unitarily equivalent to
\begin{align}\label{iso-model1''}
\left(\sbm{R_{q_i}\otimes P_i^\perp U_i + M_zR_{q_i}\otimes P_i U_i &0\\0& W_{i}}, \sbm{R_{\overline q_i}\otimes U_i^*P_i+R_{\overline q_i}M_z\otimes U_i^*P_i^\perp& 0\\0&W_{i}'}\right) \text{ on }\sbm{H^2\otimes \cF_i\\\cK_{iu}},
\end{align}where by \eqref{ExplctTuple1} the parameters $(\cF_i,\cK_{iu};P_i,U_i,W_i,W_i')$ can be chosen to be
\begin{align}\label{ExplctTuple1''}
\begin{cases}
&\cF_i=\sbm{\cD_{V_i^*}\\\cD_{V_{(i)}^*}}, \quad \cK_{iu}=\bigcap_{n\geq 0}(V_iV_{(i)})^n\cH,\quad P_i:\sbm{f\\g}\mapsto\sbm{f\\0},\\
& U_i: \sbm{D_{V_i^*}\\D_{V_{(i)}^*}V_i^*}\mapsto \sbm{D_{V_i^*}V_{(i)}^*\\D_{V_{(i)}^*}}\text{ and } (W_i,W_i')=(V_{i},V_{(i)})|_{\cK_{iu}}.
\end{cases}
 \end{align}Let us first note that by definition of $V_{(i)}$ it follows that
$$
W_{i}'=\prod_{i\neq j=1}^dW_{j}=W_{(i)}.
$$Next we note that for each $i=1,2,\dots,d$,
\begin{align*}
\cK_{iu}=\bigcap_{n\geq0}(V_iV_{(i)})^n\cH=q(i,1)q(i,2)\cdots q(i,i-1)\bigcap_{n\geq0}V^n\cH=:\cK_u,
\end{align*}where $V=V_1V_2\cdots V_d$ and we used the fact that for every $i$,
\begin{align*}
V_iV_{(i)}=V_iV_1V_2\cdots V_{i-1}V_{i+1}\cdots V_d=q(i,1)q(i,2)\cdots q(i,i-1) V.
\end{align*}We next argue that for each $i=1,2,\dots,d$, $\cF_i=\cD_{V_1^*}\oplus\cD_{V_2^*}\oplus\cdots\oplus\cD_{V_d^*}$. By the expression of $\cF_i$ as given in \eqref{ExplctTuple1''}, this will be achieved if we can show that
\begin{align}\label{aux4}
\cD_{V_{(i)}^*} \text{ is unitarily equivalent to }\oplus_{i\neq j=1}^d\cD_{V_j^*}.
\end{align}For \eqref{aux4}, we define the map $\Delta_i:\cD_{V_{(i)}^*}\to \oplus_{i\neq j=1}^d\cD_{V_j^*}$ by
\begin{align}\label{Delta}
\notag\Delta_i:D_{V_{(i)}^*}h\mapsto
D_{V_1^*}V_2^*\cdots V_{i-1}^*V_{i+1}^*\cdots V_d^*h &\oplus D_{V_2^*}V_3^* \cdots V_{i-1}^*V_{i+1}^*\cdots V_d^*h\\
&\oplus\cdots\oplus D_{V_{d-1}^*}V_{d}^*h\oplus D_{V_d^*}h.
\end{align}Using the general fact that for a contraction $T$, $\|D_Th\|^2=\|h\|^2-\|Th\|^2$, we see that
\begin{align*}
\|D_{V_1^*}V_2^*\cdots V_{i-1}^*V_{i+1}^*\cdots V_d^*h\|^2 &+ \|D_{V_2^*}V_3^* \cdots V_{i-1}^*V_{i+1}^*\cdots V_d^*h\|^2\\
&+\cdots+\|D_{V_{d-1}^*}V_{d}^*h\|^2+\|D_{V_d^*}h\|^2
\end{align*}is a telescopic sum and is equal to
\begin{align*}
\|h\|^2-\|V_1^*V_2^*\cdots V_{i-1}^*V_{i+1}^*\cdots V_d^*h\|^2=\|D_{V_{(i)}^*}h\|^2.
\end{align*}Therefore $\Delta_i$ is an isometry. We claim that
\begin{align*}
\{D_{V_1^*}V_2^*\cdots V_{i-1}^*V_{i+1}^*\cdots V_d^*h &\oplus D_{V_2^*}V_3^* \cdots V_{i-1}^*V_{i+1}^*\cdots V_d^*h\\
&\oplus\cdots\oplus D_{V_{d-1}^*}V_{d}^*h\oplus D_{V_d^*}h:h\in\cH\}=\oplus_{i\neq j=1}^d\cD_{V_j^*}.
\end{align*}We follow the same technique as used to prove Lemma \ref{justlikethat}: we show that the orthocomplement of the space on the left-hand side in $\oplus_{i\neq j=1}^d\cD_{V_j^*}$ is zero. Let $\oplus_{i\neq j=1}^d f_j\in\oplus_{i\neq j=1}^d\cD_{V_j^*}$ be such that for every $h\in\cH$,
\begin{align*}
0=\langle \oplus_{i\neq j=1}^df_j,D_{V_1^*}V_2^*\cdots V_{i-1}^*V_{i+1}^*\cdots V_d^*h &\oplus D_{V_2^*}V_3^* \cdots V_{i-1}^*V_{i+1}^*\cdots V_d^*h\\
&\oplus\cdots\oplus D_{V_{d-1}^*}V_{d}^*h\oplus D_{V_d^*}h\rangle.
\end{align*}This implies that for every $h\in\cH$
\begin{align*}
\langle h, f_d+V_df_{d-1}+\cdots+V_dV_{d-1}\cdots V_{i+1}V_{i-1}\cdots V_3f_2+V_dV_{d-1}\cdots V_{i+1}V_{i-1}\cdots V_2f_1\rangle =0,
\end{align*}which means that
$$
f_d+V_df_{d-1}+\cdots+V_dV_{d-1}\cdots V_{i+1}V_{i-1}\cdots V_3f_2+V_dV_{d-1}\cdots V_{i+1}V_{i-1}\cdots V_2f_1=0.
$$Since $D_{V_d^*}f_d=f_d$ and $D_{V_d^*}V_d=0$, we conclude by applying $D_{V_d^*}$ on the vector above that $f_d=0$. A similar analysis yields that each of the vectors $f_{d-1},\dots,f_{i+1},f_{i-1},\dots, f_1$ are zero vectors. Consequently, $\Delta_i$ is a unitary. Hence claim \ref{aux4} is proved.
\end{proof}

\begin{remark}
As in the pair case, for a tuple of $q$-commutative isometries, the BCL $q$-tuples  uniquely determine a tuple of $q$-commutative isometries in the sense that is explained for the pair case in the statement of Theorem \ref{T:uniqueness}. The proof is similar.
\end{remark}


\section{$q$-commutative unitary extension of $q$-commutative isometries}\label{S:Extension}
Just as in the commutative case, every $q$-commutative tuple of isometries can be extended to a $q$-commutative tuple of unitaries. Moreover, as the following theorem shows, this unitary extension can be made so as to have some additional structure.
\begin{thm}\label{T:Ext}
Every $d$-tuple $(X_1,X_2,\dots,X_d)$ of $q$-commutative isometric operators has a $q$-commutative unitary extension $(Y_1,Y_2,\dots,Y_d)$. Moreover, there is an extension $(Y_1,Y_2,\dots,Y_d)$ such that $Y=Y_1Y_2\cdots Y_d$ is the minimal unitary extension of $X=X_1X_2\dots X_d$.
\end{thm}
\begin{proof}
Let us suppose without loss of generality that the $q$-commutative isometric tuple $(X_1,X_2,\dots,X_d)$ is given exactly in the BCL-1 $q$-model \eqref{iso-model1'}. Consider the tuple $(Y_1,Y_2,\dots,Y_d)$ given for each $i=1,2,\dots, d$, by
\begin{align}\label{uni-model2}
Y_i=\sbm{R_{q_i}\otimes P_i^\perp U_i + M_\zeta R_{q_i}\otimes P_i U_i &0\\0& W_i} \text{ on }\sbm{L^2\otimes \cF\\\cK_u}.
\end{align}Here $L^2$ denotes the usual $L^2$ space over $\mathbb T$ with respect to the arc-length measure. It is a routine computation that the tuple $(Y_1,Y_2,\dots,Y_d)$ above is a $q$-commutative tuple of unitary operators. Moreover, it extends the model in \eqref{iso-model2} in view of the natural embedding of $(H^2\otimes \cF)\oplus \cK_u$ into $(L^2\otimes\cF)\oplus \cK_u$:
$$
\sbm{z^n\otimes \xi\\ \eta} \mapsto \sbm{\zeta^n\otimes \xi\\ \eta} \text{ for }\xi\in\cF,\;\eta\in\cK_u \text{ and }n\geq 0.
$$

For the second part of the lemma, we note that
$$X=X_1X_2\cdots X_d=X_1X_{(1)}=M_z\oplus W_{1}W_{2}\cdots W_d\text{ on $H^2(\cF_\dag)\oplus \cK_u$ }$$ and
$$Y=Y_1Y_2\cdots Y_d=M_\zeta\oplus W_1W_2\cdots W_d\text{ on $L^2(\cF_\dag)\oplus \cK_u$}.$$ Therefore it follows from the classical theory that $Y$ as above is indeed the minimal unitary extension of $X$.
\end{proof}
Let us say that a $q$-commutative tuple $(X_1,X_2,\dots,X_d)$ is doubly $q$-commutative, if in addition, it satisfies
$$
X_jX_i^*=q(i,j)X_i^*X_j\quad \mbox{for each }i,j=1,2,\dots,d.
$$As in the pair case, a $q$-commutative tuple of unitaries is automatically doubly $q$-commutative. A doubly $q$-commutative version of Theorem \ref{T:Ext} can be easily derived. 
\begin{corollary}[{\rm See also \S 6 of \cite{Jeu_Pinto}}]
Every doubly $q$-commutative tuple of isometries extends to a doubly $q$-commutative tuple of unitaries.
\end{corollary}
\begin{proof}
This follows from Theorem \ref{T:Ext} and the fact that a $q$-commutative tuple of unitaries is doubly $q$-commutative.
\end{proof}

\section{Models for $q$-commutative contractions}
Let $(T_1,T_2)$ be a pair of operators acting on a Hilbert space $\cH$. Let us call a pair $(U_1,U_2)$ of operators acting on $\cK\supset\cH$ a {\em dilation} of $(T_1,T_2)$, if
$$
T_1^mT_2^n=P_{\cH}U_1^mU_2^n|_{\cH} \text{ for every non-negative integers $m$ and $n$},
$$where $P_{\cH}$ is the orthogonal projection of $\cK$ onto $\cH$. And\^o's dilation theorem \cite{ando} states that every pair of commutative Hilbert space operators has a dilation to a pair of commutative unitary operators. Thus, a natural generalization of And\^o's dilation theorem is whether every $q$-commutative pair of contractions has a dilation to a $q$-commutative unitary operators. This question is beautifully answered in affirmative very recently in \cite{KM2019} using a commutant lifting approach. In an upcoming paper, we plan to give two constructive proofs of this {\em$q$-dilation} theorem and use the Berger--Coburn--Lebow-type model proved in this paper to consequently produce functional models for $q$-commutative pairs of contractions; the $q=1$ case is done in \cite{sauAndo}.

\end{document}